\documentclass{article}
\usepackage{tikz}
\usetikzlibrary{calc}
\usetikzlibrary{matrix}
\usetikzlibrary{positioning}

\usepackage{mypack}
\usepackage{graphicx}
\usepackage{amsmath, amsthm, amssymb}

\def\ext2et{\Ext^2_{\Un(\bX)}}

\def\I{\A^{\times}}
\def\cG{\mathfrak{G}}

\def\ms{\medskip}

\def\Fx{F^{\times}}
\usepackage{mypack, color, sgame, sgamevar, tikz}
\title{Diophantine geometry and non-abelian reciprocity laws I}
\author{Minhyong Kim}
\begin{document}
\maketitle
\begin{abstract}

We use non-abelian fundamental groups to define a sequence of higher reciprocity maps on the adelic points of a variety over a number field satisfying certain conditions in Galois cohomology. The non-abelian reciprocity law then states that the global points are contained in the kernel of all the reciprocity maps.
\end{abstract}

\ms

\ms

{\em Dedicated to John Coates on the occasion of his 70th birthday.}
\ms

{\em 
I dive down into the depth of the ocean of forms,

hoping to gain the perfect pearl of the formless...}
\ms

{\em --Tagore}

\section{Refined Hasse principles and reciprocity laws}
 Consider the Hasse-Minkowski theorem \cite{serre} for  affine conics like
\be X: \ \ \ ax^2+by^2=c\ee
stating that $X$ has a rational point in a number field $F$ if and only if it has a point in $F_v$ for all places $v$. In spite of its great elegance, even undergraduate students are normally left with a somewhat unsatisfactory sense of the statement, having essentially to do with the fact that the theorem says nothing about the locus of
\be X(F)\subset X(\A_F).\ee
There are various attempts to rectify the situation, the most successful of which might be the theory of the Brauer-Manin obstruction \cite{manin}. 
\ms

The point of view of this paper is that one should consider such problems, even for more general varieties, as that of defining a good reciprocity map.
That is, let's simplify for a moment and assume $X\simeq \Gm$ (so that existence of a rational point is not the issue). Then we are just asking about the locus of
$\Fx$ in the ideles $\I_F$ of $F$. In this regard, a description of sorts is provided by
Abelian class field theory \cite{artin-tate}, which gives us a map
\be rec^{ab}:\I_F\rTo G^{ab}_F,\ee
with the property that \be rec^{ab}(\Fx)=0.\ee
So one could well view the reciprocity map as providing a `defining equation' for $\Gm(F)$ in $\Gm(\A_F)$, except for the unusual fact that the equation takes values in a group. Because $F$ is  a number field, there is also the usual complication that the kernel of $rec^{ab}$ is not exactly equal to $\Gm(F)$. But the interpretation of the reciprocity law as a refined statement of  Diophantine geometry
is reasonable enough.
\ms

In this paper, we obtain a generalization of Artin reciprocity to an {\em iterative non-abelian reciprocity law} with coefficients in smooth varieties whose \'etale fundamental groups satisfy rather mild cohomological conditions  [Coh], to be described near the end of this section. They are satisfied, for example, by any smooth curve. Given a smooth variety $X$ equipped with a rational point $b\in X(F)$ satisfying [Coh], we define a sequence of subsets
\be X(\A_F)=X(\A_F)_1\supset X(\A_F)_1^2\supset X(\A_F)_2 \supset X(\A_F)_2^3 \supset X(\A_F)_3\supset X(\A_F)_3^4\supset \cdots \ee
and a sequence of maps
\be rec_n:X(\A_F)_n\rTo \cG_n(X)\ee
\be rec^{n+1}_n:X(\A_F)^{n+1}_n\rTo \cG^{n+1}_n(X)\ee
to a sequence $\cG_n(X), \cG^{n+1}_n(X)$of profinite  abelian groups in such a way that
$$X(\A_F)_n^{n+1}=rec_n^{-1}(0)$$
and
\be X(\A_F)_{n+1}=(rec^{n+1}_n)^{-1}(0).\ee
We visualize this as a diagram:
\be
\bd
\cdots &X(\A_F)_2^3=rec_2^{-1}(0)& \subset & X(\A_F)_2&=(rec^2_{1})^{-1}(0) & \subset &X(\A_F)^2_{1}&=rec_1^{-1}(0)  &\subset &X(\A_F)_1 \\
\cdots&\dTo^{rec_2^3}& &\dTo^{rec_2 }& & & \dTo^{rec^2_{1}} & & &\dTo^{ rec_1 } &\\
\cdots &\cG^3_2(X)& & \cG_2(X)& & & \cG^2_{1}(X) & & & \cG_1(X) &
\ed\ee
in which each reciprocity map is defined not on all of $X(\A_F)$, but only on the kernel (the inverse image of 0) of all the previous maps. We put
\be X(\A_F)_{\infty}=\cap_{i=1}^{\infty} X(\A_F)_i.\ee
The non-abelian reciprocity law then states 
\begin{thm} 
\be X(F)\subset X(\A_F)_{\infty}.\ee

\end{thm}

We give now a brief description of the groups $\cG_n$ and $\cG^{n+1}_n$. Let $\D=\pi_1(\bX, b)^{(2)}$ be the profinite  prime-to-2 \'etale fundamental group \cite{SGA1} of $\bX=X\times_{\Spec(F)}\Spec(\bar{F})$, and $\D^{[n]}$ be its lower central series defined as \be\D^{[1]}=\D\ee and
\be \D^{[n+1]}:=\overline{[\D, \D^{[n]}]},\ee
where the overline refers to the topological closure. We denote
\be \D_n:=\D/\D^{[n+1]}\ee
and
\be T_n:=\D^{[n]}/\D^{[n+1]}.\ee
Thus, we have an exact sequence
\be
1\rTo T_n \rTo \D_n\rTo \D_{n-1}\rTo 1\ee
for each $n$ turning $\D_n$ into a central extension of $\D_{n-1}$.

All of these objects are equipped with canonical actions of $G_F=\Gal(\bF/F)$. Given any topological abelian group $A$ with continuous $G_F$-action, we have the continuous Galois dual
\be D(N):=\Hom_{ct}(A, \mu_{\infty}),\ee
and the Pontriagin dual
\be A^{\vee}=\Hom_{ct}(A, \Q/\Z).\ee
(See Appendix II for details.)

With this notation, we can define the targets of the reciprocity maps $rec_n$ using continuous cohomology: 
\be \cG_n(X):=H^1(G_F, D(T_n))^{\vee}.\ee
Notice that when
$X=\Gm$, we have $T_1=\hZ^{(2)}(1)$ and $T_n=0$ for $n>1$. Thus, $D(T_1)=\oplus_{p\neq 2}\Q_p/\Z_p$ and
\be H^1(G_F, D(T_1))=\oplus_{p\neq 2} \Hom(G_F, \Q_p/\Z_p)=\oplus_{p\neq 2} \Hom(G_F^{ab}, \Q_p/\Z_p).\ee
Hence, by Pontrjagin duality, there is a canonical isomorphism
\be \cG_1((\Gm)_F)\simeq G^{ab, (2)}_F,\ee
and $rec_1$ will agree with the prime-to-2 part of the  usual reciprocity map $rec^{ab}$.

For the $\cG^{n+1}_n(X)$, we need a little more notation. 
Let $S$ be a finite set of places of $F$ and $G_F^S=\Gal(F_S/F)$ the Galois group of the maximal extension of $F$ unramified outside of $S$. We denote by $S^0$ the set of non-Archimedean places in $S$.
For a topological abelian group $A$ with $G_F^S$-action, we have the kernel of localization
\be \Sha^i_S(A):=\Ker[H^i(G^S_F, A)\rTo^{\loc_S}\prod_{v\in S^0} H^i(G_v, A)],\ee
and what we might call the {\em strict} kernel, 
\be s\Sha^i_S(A):=\Ker[H^i(G^S_F, A)\rTo^{\loc}\prod_{v\in V_F^0} H^i(G_v, A),\ee
where the localization map now goes to {\em all} non-Archimedean places in $F$. Obviously, $$s\Sha^i_S(A)\subset \Sha^i_S(A).$$ For the strict kernels, whenever $S\subset T$, the restriction maps on cohomology induce maps
\be s\Sha^i_S(A)\rTo s\Sha^i_T(A),\ee

For any finite set $M$ of odd primes, denote by $\D^M$ the maximal pro-$M$ quotient of $\D$, together with corresponding notation $[\D^{M}]^{[n]}$, $\D^M_n$, $T^M_n$. Given $M$, we consider sets of places $S$ of $F$ that contain all places lying above primes of $M$, all Archimedean places, and all places of bad reduction for $(X,b)$.

Then
\be \cG^{n+1}_n(X):=\invlim_M[\dirlim_S s\Sha^2_S(T^M_{n+1})].\ee

The conditions [Coh] are the following.
\ms

[Coh 1]: For each finite set $M$ of odd primes, $T^M_n$ is torsion-free.
\ms

[Coh 2]: For each finite set $M$ of odd primes and non-Archimedean place $v$, $H^0(G_v, T^M_n)=0$.
\ms

If we pick a place $v$, then the projection
$X(\A_F)\rTo X(F_v)$
induces the image filtration
\be X(F_v)=X(F_v)_1\supset X(F_v)_1^2\supset X(F_v)_2 \supset X(F_v)_2^3 \supset X(F_v)_3\supset X(F_v)_3^4\supset \cdots \ee
and of course, 
\be X(F)\subset X(F_v)_{\infty}:=\cap_n X(F_v)_n.\ee

\begin{conj}
When $X$ is a proper smooth curve and $v$ is an odd prime of good reduction, we have
$$X(F)=X(F_v)_{\infty}.$$
\end{conj}
This conjecture can be viewed as a refinement of the conjecture of Birch and Swinnerton-Dyer type made in \cite{BDKW}. By comparing the profinite reciprocity map here to a unipotent analogue, the computations of that paper can be viewed as evidence for (an affine analogue of) this conjecture as well. We will write more systematically about this connection and about {\em explicit} higher reciprocity laws in a forthcoming publication.

\section{Pre-reciprocity}

We will assume throughout that $X$ is a smooth variety over $F$ such that the condition [Coh] are satisfied. We will denote by $V_F$ the set of all places of $F$ and by $V_F^0$ the set of non-Archimedean places.

The maps $rec_n$ and $rec_n^{n+1}$ will  be constructed in general via non-abelian cohomology and an iterative application of Poitou-Tate duality. For this, it is important that the $G_F$-action on any fixed
$\D^M$ factors through $G^S_F=\Gal(F_S/F)$ for some finite set $S$ of places of $F$. Here, $F_S$ refers to the maximal algebraic extension of $F$ unramified outside $S$. If $X\rInto X'$ is a smooth compactifictation with a normal crossing divisor $D$ as complement, then it suffices to take $S$ large enough to satisfy the conditions that

-- $X' $ has a smooth model over $\Spec(\O_F[1/S])$;

 -- $D$ extends to a relative normal crossing divisor over $\Spec(\O_F[1/S])$;
 
 -- $b$ extends to an $S$-integral point of the model of $X$, given as the complement of the closure of $D$ in the smooth model of $X'$.
 
 -- $S$ contains $M$ and all Archimedean places of $F$.
 
 (\cite{wojtkowiak}, Theorem 2.1)
 \ms

We will be using thereby the continuous cohomology  sets and groups (Appendix I, and \cite{kim, serre2})
\be H^1(G^S_F, \D^M_n),  \ \ H^i(G^S_F, T^M_n), \ \ H^i(G^S_F, D(T^M_n)).\ee
Whenever this notation is employed, we assume that the finite set $S$ has been chosen large enough so that the $G_F$-action factors through $G^S_F$. Given any topological group
$U$ with continuous action of $G_F$, if this action factors through $G_F^S$ for some set $S$, we will call $S$ an {\em admissible set} of places. For any admissible set, we denote by $S^0$ the non-Archimedean places in $S$. 
 \ms

For each non-Archimedean place $v$ of $F$, let $G_v=\Gal(\bar{F_v}/F).$
In the following, $U$ denotes a topologically finitely-generated profinite group that is prime-to-2, in the sense that it is the inverse limit of finite groups of order prime to 2.
When $U$ has a continuous  $G_F$-action, define
\be \prod' H^i(G_v, U)=\prod'_{v\in V_F^0} H^i(G_v, U)\ee
to be the restricted direct product with respect to the subsets
\be H^i_{un}(G_v, U):=Im[H^i(G_v/I_v, U^{I_v})\rTo H^i(G_v, U)],\ee
where $I_v\subset G_v$ is the inertia subgroup. The places $v$ are taken to be all the non-Archimedean ones. Here, as in the following, if $U$ is non-abelian, we will be considering only the index $i=1$. 
Note that $H^1_{un}(G_v, U)$ is exactly the kernel of the restriction map
\be H^1(G_v, U)\rTo H^1(I_v, U).\ee

We have
\be \prod' H^i(G_v, U)=\dirlim_S [\prod_{v\in S^0}H^i(G_v, U)\times \prod_{v\in V_F^0\setminus S^0} H^i_{un}(G_v, U)]\subset \prod_{v\in V_F^0}H^1(G_v, U)\ee
as $S$ runs over admissible sets of primes. We will denote by
$\prod^S H^i(G_v, U)$ each individual set
\be \prod_{v\in S^0}H^i(G_v, U)\times \prod_{v\in V_F^0\setminus S^0} H^i_{un}(G_v, U)\ee
occurring in the limit.
We will also use the notation
\be \prod_S H^i(G_v, U)=\prod_{v\in S^0} H^i(G_v, U)\ee
and for $T\supset S$,
\be \prod_T^S H^i(G_v, U)=\prod_{v\in S^0} H^i(G_v, U)\times \prod_{v\in T^0\setminus S^0} H^1_{un}(G_v, U), \ee
so that
\be \prod^S H^i(G_v, U)=\invlim_T\prod_T^S H^i(G_v, U).\ee

For each $n\geq 2$, we have an exact sequence
\be 1\rTo T^M_n\rTo \D^M_n\rTo \D^M_{n-1}\rTo 1.\ee
of topological groups. 
By Appendix I, Lemma 4.4, the  surjection $\D^M_n\rTo \D^M_{n-1}$ is equipped with a  continuous section, so that we get a long exact sequence of continuous cohomology
\be \bd 0&\rTo &H^1(G^S_F, T^M_n)& \rTo& H^1(G^S_F, \D^M_n) &\rTo &H^1(G^S_F, \D^M_{n-1})\\
& \rTo^{\d^g_{n-1}}& H^2(G^S_F, T^M_n).&  & & & \ed\ee
Here, the superscript in `$\d^g_{n-1}$' refers to `global'.
As explained in the Appendix I,  Lemma 4.2 and Lemma 4.3, the meaning of exactness here is as follows. The group $H^1(G^S_F, T^M_n)$ acts freely  on the space $H^1(G^S_F, \D^M_n)$ and the projection \be p_{n-1}: H^1(G_F^S, \D^M_n)\rTo H^1(G_F^S,\D^M_{n-1})\ee identifies the orbit space with the kernel of the boundary map $\d^g_{n-1}$.  To check that the conditions are satisfied, note that twisting the Galois action by a cocycle for a class $c\in H^1(G_F^S,\D^M_{n-1})$ will not change the action on the graded pieces $T^M_i$, so that the condition [Coh 2] implies that $\D^M_{n-1}$ has no $G_F^S$-invariants.

Similarly, for each non-Archimedean local Galois group, we have  exact sequences
\be\bd 0&\rTo &H^1(G_v, T^M_n)& \rTo& H^1(G_v, \D^M_n) &\rTo &H^1(G_v, \D^M_{n-1})\\
& \rTo^{\d_{n-1}}& H^2(G_v, T^M_n).&  & & & \ed\ee

 For each $n$, there is a surjection
 \be (T^M_1)^{\otimes n}\rOnto T^M_n.\ee
 Thus, $T^M_n$ has strictly negative weights between $-2n$ and $-n$ as a Galois representation. 
By \cite{jannsen}, Theorem 3(b), we see that the localization
\be
H^1(G^S_F, T^M_n)\rTo \prod^S
 H^1(G_v, T^M_n)\subset \prod'  H^1(G_v, T^M_n)
\ee
is injective.
\ms
To apply \cite{jannsen}, we need to make a few remarks. Firstly, there is the simple fact that
\be T^M_n=\prod_{l\in M} T^l_n,\ee
so it suffices to consider $l$-adic representations for a fixed prime $l$.  Next, we note that \cite{jannsen}  proves the injectivity for the Galois representations $H^i(\bar{V}, \Z_l(n))/(tor)$
and $i\neq 2n$ where $V$ is a smooth projective variety. But the proof only uses the fact that this is torsion-free, finitely-generated, and of non-zero weight.

Now, by using the exact sequences (2.11) and (2.13) and an induction over $n$, we get injectivity of localiztion
\be
H^1(G^S_F, \D^M_n)\rTo \prod^S
 H^1(G_v, \D^M_n)\subset \prod'  H^1(G_v, \D^M_n)
\ee
for every $n$.

Of course, we can repeat the discussion with any admissible $T\supset S$. Using these natural localization maps, we will regard global cohomology simply as subsets of the $\prod^S$ or of $\prod'$.

For any $U$ with continuous $G^S_F$-action such that the localization map \be
H^1(G^T_F, U)\rTo \prod^T
 H^1(G_v, U)\subset \prod'  H^1(G_v, U)
\ee
is injective for all admissible $T$, define
\be E(U):=\dirlim_T  \loc(H^1(G^T_F,U))=\cup_T \loc(H^1(G^T_F,U)).\ee

For admissible $T$, there is also the partial localization
\be H^1(G_F^T, U)\rTo^{\loc_T} \prod_{T}H^1(G_v, U)).\ee

When $U$ is topologically finitely-generated abelian pro-finite group with all fintie quotients prime to 2, we have the duality isomorphism (local Tate duality, \cite{NSW}, Chapter VII.2)
\be  D: \prod_{T} H^1(G_v, U)\simeq  \prod_{T} H^1(G_v, D(U))^{\vee}\ee
that can be composed with
\be  \prod_{T} H^1(G_v, D(U))^{\vee}\rTo^{\loc^*_T}H^1(G_F^T,D(U))^{\vee}\ee
to yield a map
\be \loc^*_T\circ D: \prod_{v\in T} H^1(G_v, U)\rTo H^1(G_F^T,D(U))^{\vee}\ee
such that 
\be\Ker( \loc^*_T\circ D)=\loc_T(H^1(G_F^T, U))\ee
(Poitou-Tate duality, \cite{NSW}, Chapter VIII.6).
We denote also by
$\loc^*_T\circ D$ the map
\be \prod^TH^1(G_v, U)\rTo H^1(G_F^T,D(U))^{\vee}\ee
obtained by projecting the components in $\prod_{v\in V_F^0\setminus T^0} H^1_{un}(G_v, U)$ to zero.

When $U$ is abelian and $T'\supset T$,  these maps  fit into commutative diagrams as follows:
\be \bd
 \prod^T H^1(G_v, U)  &\rInto & \prod^{T'}H^1(G_v, U)\\
 \dTo^{\loc_T^*\circ D} & & \dTo^{\loc_{T'}^*\circ D} \\
 H^1(G^T_F, D(U))^{\vee} & \lTo^{\mbox{Inf}^*} &  H^1(G^{T'}_F, D(U))^{\vee} \ed \ee
 where the lower arrow is the dual to inflation.
 The commutativity follows from the fact that $H^1_{un}(G_v,U)$ and $H^1_{un}(G_v, D(U))$ annihilate each other under duality.
  Hence,  we get a compatible family of maps
  \be 
  \prod^T H^1(G_v, U) \rTo \invlim_{T'} H^1(G^{T'}, D(U))^{\vee}=H^1(G_F, D(U))^{\vee}\ee
Taking the union over $T$, we then get
\be
 prec(U): \prod' H^1(G_v, U)\rTo \invlim_T H^1(G_{T}, D(U))^{\vee}=H^1(G_F, D(U))^{\vee}
\ee
(For `pre-reciprocity'.) 
According to Appendix II (5.13),
\begin{prop} When $U$ is topologically finitely-generated abelian pro-finite group with all fintie quotients prime to 2, then
\be \Ker(prec(U))=E(U).\ee
\end{prop}

One distinction from the appendix is that our product runs only over non-Archimedean places. However, because we are only considering prime-to-2 coefficients, the local $H^1$ vanishes as all Archimedean places.
The goal of this section, by and large, is to generalise this result to the coeffiecients $\D^M_n$, which are non-abelian.

In addition to the exact sequences (2.19) and (2.21), we have  exact sequences with  restricted direct products
\be \bd 0&\rTo &\prod' H^1(G_v, T^M_n)& \rTo& \prod'  H^1(G_v, \D^M_n) &\rTo^{p_{n-1}} & \prod'  H^1(G_v, \D^M_{n-1})\\
& \rTo^{\d_{n-1}}& \prod'  H^2(G_v, T^M_n)&  & & & \ed\ee
making the middle term of the first line a $\prod' H^1(G_v, T^M_n)$-torsor over the kernel of $\d$. To see this, let $S$ be an admissible set of primes. Then the $G_v$-action for $v\notin S$ factors through $G_v/I_v$, so that we have an exact sequence
\be \bd 
 0&\rTo &H^1(G_v/I_v, T^M_n)& \rTo& H^1(G_v/I_v, \D^M_n) &\rTo &H^1(G_v/I_v, \D^M_{n-1})\\
& \rTo^{\d_{n-1}}& H^2(G_v/I_v, T^M_n)&  & & & \ed\ee
and hence, an exact sequence
\be \bd 
 0\rTo &\prod_{v\in S^0}H^1(G_v, T^M_n)\times  \prod_{v\in V_F^0\setminus S^0} H^1_{un}(G_v/I_v, T^M_n)& \rTo& \prod_{v\in S}H^1(G_v, \D^M_n)\times  \prod_{v\in V_F^0\setminus S^0}  H^1_{un}(G_v/I_v, \D^M_n)\\
 \rTo &\prod_{v\in S}H^1(G_v, \D^M_{n-1})\times  \prod_{v\in V_F^0\setminus S^0}  H^1_{un}(G_v/I_v, \D^M_{n-1}) & \rTo^{\d_{n-1}}& \prod_{v\in S}H^2(G_v, T^M_n)\times  \prod_{v\in V_F^0\setminus S^0}  H^2_{un}(G_v/I_v, T^M_n)   \ed\ee
Taking the direct limit over $S$ gives us the exact sequence with restricted direct products.

\ms

In the following, various local, global, and product boundary maps will occur. In the notation, we will just distinguish the level and the global boundary map, since the domain should be mostly clear from the context.

We go on to define a sequence of pre-reciprocity maps as follows.
First, we let
\be prec_1 :=\invlim_M prec(\D^M_1): \invlim_M \prod' H^1(G_v, \D^M_1)\rTo \invlim_M H^1(G_F, D(\D^M_1))^{\vee}=H^1(G_F, D(\D_1))^{\vee}\ee
as above. The kernel of $prec_1$ is exactly $E_1:=\invlim_M E(\D_1^M).$
For $x\in  E_1$, define
\be prec_1^2(x):=\d^g_1(x)\in \invlim_M \dirlim_T H^2(G^T_F, T^M_2)\ee
and \be E_1^2:=\Ker(prec_1^2).\ee
Given  $x\in E_1^2$ we will denote by $x_M$ the projection to \be [E_1^2]_M:=\Ker [\d^g_1| E(\D^M_1)].\ee
We will be considering various inverse limits over $M$ below, and using subscripts $M$ in a consistent fashion.

Now define
\be W(\D^M_2)\subset \prod' H^1(G_v, \D^M_2)\ee
to be the inverse image of $[E_1^2]_M$ under the projection map
\be  p_1: \prod' H^1(G_v, \D^M_2)\rTo \prod' H^1(G_v, \D^M_1), \ee  which is, therefore, a
$\prod' H^1(G_v, T^M_2)$-torsor over $[E_1^2]_M$ . 

Consider the following diagram:

\be \bd
E(T^M_2) & \rInto & \prod' H^1(G_v, T^M_2) \\
\dInto & & \dInto \\
E( \D^M_2) & \rInto & W(\D^M_2) \\
\dTo & & \dTo \\
[E_1^2]_M& = & [E_1^2]_M
\ed\ee
We see with this that $E(\D^M_2)$ provides a reduction of structure group for $W(\D^M_2)$ from
$\prod' H^1(G_v, T^M_2)$ to $E(T^M_2)$.
\ms

Choose a set-theoretic splitting
\be s_1: [E_1^2]_M \rTo E(\D^M_2)\ee
of the torsor in the left column. We then use this `global' splitting to define
\be prec_2^M: W(\D^M_2)\rTo H^1(G_F, D(T^M_2))^{\vee}\ee
by the formula
\be prec^M_2(x)= prec(T^M_2)(x-s_1(p_1(x)))\ee
Here, we denote by $x-s_1(p_1(x))$ the unique element $z\in \prod' H^1(G_v, T^M_2)$ such that
$x=s_1(p_1(x))+z.$

Because $E(T^M_2)$ is killed by $prec(T^M_2)$, it is easy to see that
\begin{prop}
$prec_2^M$ is independent of the splitting $s_1$.
\end{prop}
Now define
\be W_2:=\invlim_M  W_2(\D_2^M)\ee
and
\be prec_2:=\invlim_M prec_2^M: W_2 \rTo \invlim_M H^1(G_F, D(T^M_2))^{\vee}=H^1(G_F, D(T_2))^{\vee}.\ee

In general, define \be E_n:=\invlim_M E(\D_n^M)\ee
and
\be prec_n^{n+1}:=\d^g_n:E_n \rTo \invlim_M\dirlim_TH^2(G_T, T^M_{n+1}).\ee
Then define
\be E^{n+1}_n=\Ker(\d^g_n)\subset E_n,\ee
and
\be W(\D^M_{n+1})=p_n^{-1}([E^{n+1}_n]_M),\ee
where $[E^{n+1}_n]_M=\Ker(\d^g|E(\D^M_n))$.
Use a splitting $s_{n}$ of 
\be \bd E(\D^M_{n+1})&\rTo&
[E_n^{n+1}]_M\ed\ee to define
\be prec^M_{n+1}: W(\D^M_{n+1})\rTo H^1(G_F, D(T^M_{n+1}))^{\vee}\ee
via the formula
\be prec^M_{n+1}(x)= prec(T^M_{n+1})(x-s_{n}(p_{n}(x))).\ee
Once again, because $E(T^M_{n+1})$ is killed by $prec(T^M_{n+1})$, we get \begin{prop}
$prec_n^M$ is independent of the splitting $s_{n}$.

\end{prop}
Finally,  define
\be W_{n+1}:=\invlim_M W(\D^M_{n+1})\ee
and
\be prec_{n+1}=\invlim_M prec_{n+1}^M: W_{n+1}\rTo \invlim_M H^1(G_F, D(T^M_{n+1}))^{\vee}=
H^1(G_F, D(T_{n+1}))^{\vee}.\ee

Then we finally have the following generalisation of Proposition 2.2.
\begin{prop}
\be Ker(prec_{n+1})=E_{n+1}\ee
\end{prop}
\begin{proof}
We have seen this already for $n=1$. 
Let $x\in \Ker(prec_{n+1})$ and $x_M$ the projection to $\Ker(prec_{n+1}^M)$. 
It is clear from the definition that $E(\D^M_{n+1})\subset Ker(prec^M_{n+1})$.
On the other hand, if $prec^M_{n+1}(x)=0$, then $y_M=x_M-s_{n}(p_{n}(x))\in E(T^M_{n+1}),$ by Proposition 2.2. Hence,
$x_M=y_M+s_{n}(p_{n}(x))\in E(\D^M_{n+1}).$ Since this is true for all $M$, $x\in E_{n+1}=\invlim E(\D_{n+1}^M)$.

\end{proof}

\section{Reciprocity}
Recall the product of the local period  maps
\be j^M_n: X(\A_F)\rTo \prod' H^1(G_v, \D^M_n),\ee
\be x\mapsto (\pi_1^{et}(\bX; b, x_v)^M_n)_v.\ee
Here,$$\pi_1^{et}(\bX; b, x_v)^M_n:=\pi_1^{et}(\bX; b, x_v)\times_{\pi_1^{et}(\bX,b)} \D^M_n=[\pi_1^{et}(\bX; b, x_v)\times \D^M_n]/\pi_1^{et}(\bX,b) $$
 are torsors for $\D^M_n$ with compatible actions of $G_v$, and hence, define classes in $H^1(G_v, \D^M_n)$. When $v\notin S$ for $S$ admissible and
$x_v\in X(\O_{F_v})$, then this class belongs to $H^1_{un}(G_v, \D^M_n)$ ( \cite{wojtkowiak}, Prop. 2.3). Therefore,
$ (\pi_1^{et}(\bX; b, x_v)^M_n)_v$ defines a class in $\prod'H^1(G_v, \D^M_n)$.
 (This discussion is exactly parallel to the unipotent case \cite{kim2, KT}.)
 Clearly, we can then take the limit over $M$, to get the period map
 \be j_n: X(\A_F)\rTo \invlim_M\prod' H^1(G_v, \D^M_n).\ee
 \ms
 
The reciprocity maps will be  defined by\be
rec_n(x)=prec_n(j_n(x)),\ee  
and
\be rec_n^{n+1}(x)=prec^{n+1}_n(j_n(x)).\ee
Of course, these maps will not be defined on all of $X(\A_F)$. As in the introduction, define
\be X(\A_F)_1^2=Ker (rec_1).\ee
Then for $x\in X(\A_F)_1^2$, $j_1(x)\in E(\D_1)$, and hence, $prec_1^2$ is defined on $j_1(x)$. Thus,
$rec_1^2$ is defined on $X(\A_F)_1^2$.  Now define
\be X(\A_F)_2:=\Ker(rec_1^2).\ee
Then for $x\in X(\A_F)_2$, $j_1(x)\in E^2_1$, so that $j_2(x)\in W_2$.
Hence, $prec_2$ is defined on $j_2(x)$, and $rec_2$ is defined on $X(\A_F)_2$.

In general, the following proposition is now clear.
\begin{prop} Assume we have defined \be X(\A_F)_1^2\supset X(\A_F)_{2}\supset \cdots \supset X(\A_F)_{n-1}^{n}\supset X(\A_F)_n\ee as the iterative kernels of
$rec_1, rec^2_1, \cdots,  rec_{n-1}, rec_{n-1}^n$. Then, $j_n(x)\in W_n$, so that
$rec_n=prec_n\circ j_n$ is defined on $X(\A_F)_{n}$ and $rec_n^{n+1}$ is defined on $\Ker (rec_n)$.
\end{prop}

Note that $prec^n_{n-1}$ takes values in $\invlim_M \dirlim_TH^2(G^T_F, T^M_n).$
However, $j_{n-1}(x)$ lifts to $j_n(x)$, and hence, is clearly in the kernel of
$\d_n$. Therefore, 
\be prec^n_{n-1}(j_{n-1}(x))\in \invlim_M\dirlim_T s\Sha^2_T(T^M_n)=:\cG^n_{n-1}(X)\ee
for all $x\in X(\A_F)_{n-1}$.

The global reciprocity law now follows immediately from the commutativity of the diagram
\be \bd
X(F)& \rInto & X(\A_F) \\
\dTo & & \dTo \\
E(\D^M_n)=\dirlim_T H^1(G_F^T, \D^M_n) & \rInto & \prod' H^1(G_v, \D^M_n)
\ed\ee
for each $M$.

\ms

To check compatibility with the usual reciprocity map for $X=\Gm$ note  that the map
\be F_v^*\rTo^{\k} H^1(G_v, \hZ(1)^{(2)})\rTo^{D} H^1(G_v, \oplus_{p\neq 2} \Q_p/\Z_p)^{\vee}=G_v^{ab, (2)}\ee
is the local reciprocity map (\cite{NSW}, corollary (7.2.13), with the natural modification for the prime-to-2 part). Here, $\k$ is the map given by Kummer theory, while $D$ is local duality as before. Furthermore, the localization
\be H^1(G_F, \oplus_{p\neq 2} \Q_p/\Z_p)\rTo^{\loc_v} H^1(G_v, \oplus_{p\neq 2} \Q_p/\Z_p)\ee
is dual to the map,
\be G^{(2)}_v\rTo G^{(2)}_F\ee
induced by $\bF\rInto \bF_v$,
so that the dual of localization
\be G_v^{ab,(2)}=H^1(G_v, \oplus_{p\neq 2} \Q_p/\Z_p)^{\vee}\rTo^{\loc_v^*}H^1(G_F, \oplus_{p\neq 2} \Q_p/\Z_p)^{\vee}=G^{ab,(2)}_F\ee
is simply the natural map we started out with. Since the global reciprocity map is the sum of local reciprocity maps followed by the inclusion of decomposition groups, we are done.

\section{Appendix I: A few lemmas on non-abelian cohomology}
We include here some basic facts  for the convenience of the reader.
\ms

Given a continuous action
\be
\rho: G\times U\rTo U\ee
of
topological group $G$  on a topological group $U$, we will only need $H^0(G,U)$ and $H^1(G,U)$ in general. Of course $H^0(G,U)=U^{\rho}\subset U$ is the subgroup of $G$-invariant elements. (We will put the homomorphism $\rho$ into the notation or not depending upon the needs of the situation.) Meanwhile,
we define
\be H^1(G,U)=U\bsl Z^1(G,U).\ee
Here, $Z^1(G,U)$ consists of the 1-cocycle, that is, continous maps $c:G\rTo U$ such that
\be c(gh)=c(g)gc(h),\ee
while the $U$ action on it is given by
\be (uc)(g):=uc(g)g(u^{-1}).\ee
We also need $H^2(G,A)$ for $A$ abelian defined in the usual way as the 2-cocycles, that is, continuous functions $c: G\times G\rTo A$ such that
\be gc(h,k)-c(gh, k)+c(g,hk)-c(g,h)=0,\ee
modulo the subgroup of elements of the form
\be df(g,h)=f(gh)-f(g)-gf(h)\ee
for $f:G\rTo A$ continuous.
Any $H^i(G,U)$ defined in this way is pointed by the class of the constant map
$G^n\rTo e\in U$, even though it is a group in general only for $U$ abelian.
We denote by $[c]$ the equivalence class of a cocycle $c$.
\ms

Given a 1-cocycle $c\in Z^1(G,U)$, we can define the twisted action
\be \rho_c: G\rTo \Aut(U)\ee
as
\be \rho_c(g)u=c(g)\rho(g)(u)c(g)^{-1}.\ee
The isomorphism class of this action depends only on the equivalence class $[c]$.
\ms

Given an exact sequence
\be 0\rTo^i A\rTo B\rTo^q C\rTo 0\ee
of topological groups with $G$ action such that the last map admits a continuous splitting (not necessarily a homomorphism)  and $A$ is central in $B$, we get the exact sequence 
\be 0\rTo H^0(G,A) \rTo H^0(G,B)\rTo H^0(G,C)\rTo\ee
\be \rTo H^1(G,A) \rTo H^1(G,B)\rTo H^1(G,C)\rTo H^2(G,A)\ee
of pointed sets, in the sense that the image of one map is exactly the inverse image of the base-point for the next map (\cite{serre2}, Appendix to Chapter VII).
\ms

But there are several bits more of structure. Consider the fibers of the map
\be i_*: H^1(G,A)\rTo H^1(G,B).\ee
The group $H^0(G,C)$ will act on $H^1(G,A)$ as follows. For $c\in H^0(G,C)$, choose a lift to
$b\in B$. For $x\in Z^1(G,A)$, let
$(bx)(g)=bx(g)g(b^{-1}).$ Because $c$ is $G$-invariant, this take values in $A$, and defines a cocycle. Also, a different choice of $b$ will result in an equivalent cocycle, so that the action on $H^1(G,A)$ is well-defined. From the definition, the $H^0(G,C)$-action preserves the fibers of $i_*$. Conversely, if $[x]$ and $[x']$ map to the same element of $H^1(G,B)$, then there is a $b\in B$ such that
$x'(g)=bx(g)g(b^{-1})$ for all $g\in G$. But then by applying $q$, we get
$0=q(b)g(q(b)^{-1}),$ that is, $q(b)\in H^0(G, C).$  We have shown:
\begin{lem}
The fibers of $i_*$ are exactly the $H^0(G,C)$-orbits of $H^1(G, A)$.
\end{lem}

We can say more. Given $x\in Z^1(G, A)$ and $y\in Z^1(G,B)$, consider the map
\be (xy)(g):=x(g)y(g).\ee
This is easily seen to be in $Z^1(G,B)$ and defines an action of $H^1(G,A)$ on $H^1(G,B)$.
Fix an element $[y]\in H^1(G,B)$. We work out its stabilizer.  We have
$ [x][y]=[xy]=[y]$ if and only if there is a $b\in B$ such that $x(g)y(g)=by(g)g(b^{-1}).$
By composing with $q$, we get
\be qy(g)=q(b)qy(g)g(q(b)^{-1})\ee
or
\be
qy(g)g(q(b))qy(g)^{-1}=q(b).\ee
This says that $q(b)$ is invariant under the $G$-action $\rho_{qy}$ given by
\be c \mapsto qy(g)g(c)qy(g)^{-1}.\ee
In fact, it is easy to see that 
\begin{lem} Suppose $C^{\rho_{z}}=0$ for all $[z]\in H^1(G,C)$. Then
$H^1(G,A)$ acts freely on $H^1(G,B)$.
\end{lem}
\begin{proof}
From the paragraph before the statement, we get that $q(b)=0$, and hence,
$b\in A$. But then, $x(g)y(g)=bg(b^{-1})y(g)$ for all $g$, from which we deduce that
$x(g)=bg(b^{-1})$ for all $g$, so that $[x]=0$.\end{proof}

On the other hand,
\begin{lem}
The action of $H^1(G,A)$ is transitive on the fibers of $q_*: H^1(G,B)\rTo H^1(G,C).$
\end{lem}
\begin{proof}
The action clearly preserves the fiber. Now suppose
$[qy]=[qy']\in H^1(G,C)$. Then there is a $c\in C$ such that.
\be qy'(g)= c qy(g)g(c^{-1})\ee
for all $g$. We can lift $c$ to $b\in B$, from which we get
\be y'(g)=x(g) by(g)g(b^{-1}) \ee
for some $x(g)\in A$. This equality can be used to show that
\be x:G\rTo A\ee
is a cocycle, and $[y']=[x][y]$.\end{proof}

The existence of the continuous splitting of exact sequences that we need for applying the results above always holds in the profinite case. 
\begin{lem}Suppose we have an exact sequence of profinite groups
\be 0\rTo A\rTo B\rTo^pC\rTo 0\ee
where all maps are continuous. Suppose $B=\invlim_j B_j$, where the $j$ run over natural numbers. Then there is a continuous section to the map $B\rTo C$.
\end{lem}
\begin{proof}

If $B=\invlim B_j$, by replacing each $B_j$ with the image of $B$ if necessary, we can assume all the maps in the inverse system are surjective. Furthermore, if $A_j$ is the image of $A$ in $B_j$,  and $C_j=B_j/A_j$, one gets
$A=\invlim A_j$ (since $A$ is closed in $B$) and $C=\invlim C_j$.  That is, the exact sequence of profinite groups can be constructed as an inverse limit of  exact sequences
\be 0\rTo A_j\rTo B_j\rTo^{p_j} C_j\rTo 0\ee
indexed by the same category in such as way that all the transition maps \be A_{i}\rTo A_j, \ B_{i}\rTo B_j,  \ C_{i}\rTo C_j\ee
are surjective.
From the commutative diagram
\be \bd B_{i}& \rTo & C_{i}\\
\dTo^f & & \dTo^g \\
B_{j}& \rTo & C_j\ed \ee
we get the commutative diagram
\be \bd B_{i}& \rTo^h & B_j\times_{C_j} C_{i}\\
 &\rdTo & \dTo\\
 & & C_{i} \ed\ee
We claim that the map $h$ is surjective.
 To see this, let $c_j\in C_j$ and  $b_j\in B_j$ map to $c_j$. We need to check that $f^{-1}(b_j)$ surjects onto $g^{-1}(c_j)$. Let $c_{i}\in g^{-1}(c_j)$. Choose $b'_{i}\in B_{i}$ mapping to $c_{i}$ and let $b_j'=f(b'_{i})$. Since $b_j'$ and $b_j$ both map to $c_j$, there is an $a_j\in A_j$ such that $b_j'=b_j+a_j$. Now choose $a_{i}$ mapping to $a_j$ and put 
 $b_{i}=b_{i}'-a_{i}$. Then $b_{i}\in f^{-1}(b_j)$ and it still maps to $c_{i}$.
This proves the claim.

 For any fixed $j$, suppose we've chosen a section $s_j$ of $B_j\rTo C_j$. Then \be s_j\circ g:C_{i}\rTo B_j\ee defines a section of
 \be B_j\times_{C_j} C_{i}\rTo C_{i}.\ee
 This section can then be lifted to a section $s_{i}$ of $B_{i}\rTo C_{i}$. Thereby, we have constructed  a diagram of sections
 \be \bd C_{i}& \rTo^{s_{i}}& B_{i} \\
 \dTo & & \dTo \\
 C_{j}& \rTo^{s_j}& B_{j} \ed \ee
 
By composing $s_j$ with the projection $g_j: C\rTo C_j$, we have a compatible sequence of maps
\be C\rTo^{f_j=s_j\circ g_j} B_j\ee
such that $p_j\circ f_j=g_j$. Thus, we get a continuous map $f:C\rTo B$ such that $p\circ f=Id$.

\end{proof}

Clearly, a continuous section must exist in circumstances more general than countably ordered inverse limits, but we have just recalled this case since it is all we will need. This applies for example when $B$ is the pro-$M$ completion of a finitely-generated group: For every $n$, we can let $B(n)\subset B$ be the intersection of open subgroups of index $\leq n$. This is a characteristic subgroup, and still open. So the quotients defining the inverse limit can be taken as
$B/B(n)$.

\section{Appendix II: Some complements on duality for Galois cohomology}

When $A$ is  topological abelian group, $A^{\vee}$ denotes the continuous homomorphisms to the discrete group $\Q/\Z$. Thus, in the profinite case of $A=\invlim A_i$,
\be A^{\vee}=\dirlim_i \Hom(A_i,\Q/\Z)\ee
with the discrete topology.  If $A=\dirlim_m A[m]$ is a discrete torsion abelian group, then
\be A^{\vee}=\invlim \Hom (A[m], \Q/\Z)\ee
with the projective limit topology.
Meanwhile, if $A$ has a continuous action of the Galois group of a local or a global field, then
$D(A)$ denotes the continuous homomorphisms to the discrete group \be \mu_{\infty}=\dirlim_m \mu_m\ee
with Galois action. As far as the topological group structure is concerned, $D(A)$ is of course the same as $A^{\vee}$.

We let $F$ be a number field and $T$ a finite set of places of $F$ including the Archimedean places. We denote by $G_F$ the Galois group $\Gal(\bF/F)$ and by $G^T_F=\Gal(F_T/F)$ the Galois group of the maximal extension of $F$ unramified outside $T$. Let $v$ be a place of $F$ and $G_v=\Gal(\bF_v/F_v)$ equipped with a choice of homomorphism $G_v\rTo G_F\rTo  G^T_F$ given by the choice of an embedding $\bF\rInto \bF_v.$

In the following $A$ (with or without Galois action) will be in the abelian sub-category of all abelian groups generated by topologically finitely-generated pro-finite abelian groups and torsion groups $A$ such that   $A^{\vee}$ is topologically finitely-generated.

We have local Tate duality
\be H^i(G_v, A)\simeq^{D} H^{2-i}(G_v, D(A))^{\vee}.\ee
We also use the same letter $D$ to denote
 the product isomorphisms
\be \prod_{v\in T'} H^i(G_v, A)\simeq^{D} \prod_{v\in T'}H^{2-i}(G_v, D(A))^{\vee}\ee 
for any indexing set $T'$.

Let $\Sha^i_T(A)$ be the kernel of the localization map
\be \Sha^i_T(A):=\Ker[H^i(G^T_F, A)\rTo^{\loc_T} \prod_{v\in T}H^i(G_v, A)]\ee
and $Im^i_T(A)$, the image of the localization map
\be Im^i_T(A):=Im [H^i(G^T_F, A)\rTo^{\loc_T} \prod_{v\in T}H^i(G_v, A)].\ee
Assume now that $A=\invlim A_n$, where for all $n$, all primes dividing the order of $A_n$ lie below $T$. According to Poitou-Tate duality, we have an isomorphism
\be \Sha^i_T(A)\simeq \Sha^{2-i}_T(D(A))^{\vee},\ee
and an exact sequence
\be H^i(G^T_F, A)\rTo \prod_{v\in T} H^i(G_v, A)\rTo^{\loc_T^*\circ D} H^{2-i}(G^T_F , D(A))^{\vee}.\ee
Note that this is usually stated for finite coefficients. But since all the groups in the exact sequence
\be H^i(G^T_F, A_n)\rTo \prod_{v\in T} H^i(G_v, A_n)\rTo^{\loc_T^*\circ D} H^{2-i}(G^T_F, D(A_n))^{\vee}\ee 
 are finite, we can take an inverse limit to get the exact sequence above.
 
 If $T'\supset T$,  since all the inertia subgroups $I_v\subset G_v$ for $v\notin T$ act trivially on $A$, we have 
\be Im^1_{T'}(A)\cap [\prod_{v\in T}H^1(G_v, A)\times \prod_{v\in T'\setminus T}H^1(G_v/I_v,A)]=Im^1_T(A).\ee
 In particular, we have an exact sequence
 \be H^1(G^T_F,A)\rTo^{\loc_{T'}} \prod_{v\in T}H^1(G_v, A)\times \prod_{v\in T'\setminus T}H^1(G_v/I_v,A)\rTo^{\loc_{T'}^*\circ D} H^1(G^{T'}_F, D(A))^{\vee}.\ee
 Taking an inverse limit over $T'$, we get an exact sequence
 \be H^1(G^T_F,A)\rTo^{\loc} \prod_{v\in T} H^1(G_v, A)\times \prod_{v\notin T} H^1(G_v/I_v, A)\rTo^{\loc^*\circ D} H^1(G_F, D(A))^{\vee}.\ee 
\ms

\ms

{\bf Acknowledgements} It is a great pleasure and an honour to dedicate this humble paper to John Coates in gratitude for all his kind guidance and support over the years. I am also grateful to Henri Darmon, Jonathan Pridham, Romyar Sharifi, Junecue Suh, and Andrew Wiles for many helpful discussions on reciprocity laws.

\ms

 {\footnotesize Mathematical Institute, University of Oxford, Woodstock Road, Oxford, OX2 6GG, and Department of Mathematical Sciences, Seoul National University, 1 Gwanak-ro Gwanak-gu, Seoul 151-749, Korea. email: minhyong.kim@maths.ox.ac.uk}
\end{document}